\documentclass[11pt,a4paper]{article}
\usepackage{epsfig}
\usepackage[numbers]{natbib}
\usepackage{amsmath, amsthm, amssymb}
\usepackage{epsfig}
\usepackage[dvipdfm]{hyperref}
\usepackage[margin=3cm, top=3.5cm, bottom=3.5cm]{geometry}

\newcommand\blfootnote[1]{%
  \begingroup
  \renewcommand\thefootnote{}\footnote{#1}%
  \addtocounter{footnote}{-1}%
  \endgroup
}
\newtheorem{theorem}{Theorem}[section]
\newtheorem{lemma}{Lemma}[section]

\numberwithin{equation}{section}
\everymath{\displaystyle}
\allowdisplaybreaks

\title{Green's function for the  viscoelastic and isotropic half-space}
\date{}
\author{
Tsviatko Rangelov\thanks{Institute of Mathematics and Informatics, Bulgarian Academy of Sciences, 1113, Sofia, Bulgaria}
 \and Petia S. Dineva\thanks{Institute of Mechanics, Bulgarian Academy of Sciences, 1113, Sofia, Bulgaria}
 \and George D. Manolis\thanks{Department of Civil Engineering, Aristotle University,
 GR-54124, Thessaloniki, Greece}
 \footnotemark[1]
}

\begin{document}
\maketitle              
\blfootnote{Corresponding author: T. Rangelov, rangelov@math.bas.bg}

\begin{abstract}
\noindent
A dynamic 3D Green's function for the homogeneous, isotropic and viscoelastic (of the Zener type) half-space is derived in a closed form.  The results obtained here can be used as either stand-alone solutions for simple problems or in conjunction with a boundary integral equation formulations to account for complex boundary conditions. In the later case, mesh-reducing boundary element formulations can be constructed as an alternative method for numerical implementation purposes.
\end{abstract}


{\bf Math. Subj. Class.} 35Q74, 35J08

\section{Introduction}
\label{sec1}
Wave propagation in either the elastic full- or half-space is the cornerstone of elastodynamics and forms the basis for solving various categories of wave scattering problems in physics, solid mechanics, material science, geophysics and earthquake engineering. More specifically, the construction of fundamental solutions is a key step for solving the following categories of problems:
\begin{itemize}
\item[(a)] Wave propagation in a half-space with surface relief and containing layers, cracks, faults, cavities and inclusions;
\item[(b)] soils-structure interaction problems for half-space with different type of engineering construction such as foundations, buildings, bridges, dams, etc.;
\item[(c)] non-destructive testing evaluation of composite materials with different type of defects; and
\item[(d)] waveform inversion in geophysics.
\end{itemize}
After several decades of development, the Boundary Element Method (BEM) has found a firm footing in the area of numerical methods for  solution of partial differential equations describing elastic wave motion in semi-infinite domains. Comparing with the more popular, domain-based numerical techniques such as the Finite Element Method (FEM) and the Finite Difference Method (FDM), the BEM is a boundary-only approach, meaning that the corresponding boundary-value problem (BVP) is described by integral equations formulated along the periphery of the object  under consideration. The conventional BEM is based on the use of a fundamental solution to the governing equation, while the non-conventional BEM  uses the half-space Green`s function, which  in addition to the fundamental solution, satisfies traction-free boundary conditions, see \cite{Do93} and \cite{MDRW17}. Fundamental solutions and/or Green`s functions serve as kernels in integral equation formulations and are redefined for a source $x$ (where the load is applied) plus a receiver $\xi$ (where the response is measured) configuration. The key role played by BEM formulations in computational mechanics is to reduce a BVP formulated by the governing partial differential equation together with boundary and initial conditions into a system of integral equations through the use of reciprocal theorems. It is for this reason that the recovery of fundamental solutions and/or Green`s function in analytical form is so important. A BEM formulation based on the half-space Green`s function is more desirable because it becomes unnecessary to discretize the free surface and this results in a numerical scheme with reduced computational time and memory storage requirements.

From a historical perspective, the problems involving half-space Green`s functions concern the early twentieth century, when in \cite{La04}  and later in \cite{Mi60}  have been investigated the response of a homogeneous, elastic and isotropic half-space subjected to concentrated vertical and horizontal loads at either its surface or its interior. A detailed state-of-the-art regarding the literature on half-space Green`s functions can be found in \cite{Ka06} and  \cite{DMW19}. Furthermore, frequency-dependent half-space Green`s functions were derived in \cite{Jo74}, \cite{Pa87}, \cite{TAG01} and \cite{YP16}. In more detail, the analytical expressions for the singular displacements along the free surface of the half-space are represented in terms of semi-infinite integrals. In most cases, these integrals cannot be evaluated analytically, while their numerical solution is difficult and time-consuming due to the oscillating behavior of the integrands and because of the presence of singularities in these Green`s functions.

Obviously, the use of half-space Green`s functions is unavoidable when considering advanced, high-performance BEM formulations for problems involving semi-infinite media. This fact motivated the authors to derive an analytical solution for a half-space Green`s function for a homogeneous, viscoelastic and isotropic  medium  that is convenient for implementation in the BEM-based models.
\section{Mathematical statement of the problem}
\label{sec2}
The aim of the present work is to derive a dynamic Green's function for viscoelastic, isotropic half-space under a point force in space and time that satisfies the traction-free boundary condition along its boundary.

Dynamic equilibrium equations and boundary conditions in the half-space $\mathbb{R}^3_-=\{x=(x_1,x_2,x_3), x_3<0\}$ are considered. At first, the elasticity tensor is denoted as  $C_{jkpq}=\lambda\delta_{jk}\delta_{pq}+\mu(\delta_{jp}\delta_{kq}+\delta_{jq}\delta_{kp})$, where $\lambda, \mu$ are the Lam\'e constants, with $\mu>0, \lambda+\mu>0$ and indexes vary $1, 2, 3$.

Next, we introduce the Zener constitutive relation for viscoelasticity with fractional derivatives of order $\alpha$ in the frequency domain, see \cite{Ma10}, \cite{RDM18} and  \cite{MM16}:
\begin{equation}
\label{eq001} (1+p(i\omega)^\alpha)\sigma_{ij}(x,\omega)=C_{ijkl}(1+q(i\omega)^\alpha)\varepsilon_{kl}(x,\omega).
\end{equation}
In the above, $p\geq0$, $q\geq0$, $\alpha\in[0,1]$ are the model coefficients, $\sigma_{ij}(x,\omega)$ is the stress tensor and $\varepsilon_{kl}(x,\omega)$ is the strain tensor.

The Green`s function $g(x,\xi,\omega)$ for this problem must satisfy the following boundary-value problem:
\begin{equation}
\label{eq1a}
L(g)=\sigma_{ijm,i}(x,\xi,\omega)+\rho\tilde{\omega}^2g_{jm}(x,\xi,\omega)=-\delta_{jm}\delta(x-\xi), \ \ x,\xi\in \mathbb{R}^3_-,
\end{equation}
\begin{equation}
\label{eq1b}
T^g_{jm}(x,\xi,\omega)=0  \hbox{ on } x_3=0,
\end{equation}
In the above, $\rho$ is the material density, $\tilde{\omega}^2=\frac{1+p(i\omega)^\alpha}{1+q(i\omega)^\alpha}\omega^2$ is the square of the damped frequency and $\sigma_{ijm}(x,\xi,\omega)=C_{ijkl}g_{km,l}(x,\xi,\omega)$ is the viscoelastic stress tensor. Furthermore, $T^g_{jm}=\sigma_{ijm}n_i=\sigma_{3jm}$ is the traction tensor with $n=(0,0,1)$ the outward pointing, normal vector to $x_3=0$, and $\delta(x,\xi)$ is Dirac`s delta function. Note that comma  denotes partial differentiation with respect to the space coordinates, the summation convention for repeated indexes is implied and Green`s tensor $g(x,\xi,\omega)$  satisfies the Sommerfeld radiation condition along lines parallel to $\{x_3=0\}$.

We start with a fundamental solution $u^\ast(x,\xi,\omega)$, i.e., a solution of Eq. (\ref{eq1a}) and recover in Section \ref{sec3} a smooth function $w(x,\xi\omega)$ which solves the homogeneous part of Eq. (\ref{eq1a}), such that $g(x,\xi,\omega)=u^\ast(x,\xi,\omega)+w(x,\xi,\omega)$ satisfies the boundary conditions in Eq. (\ref{eq1b}).
\section{Fundamental solution}
\label{sec3}

Assume that matrix-valued function $u^\ast(x,\xi,\omega)$  is a fundamental solution  of
Eq.  (\ref{eq1a}), i.e.,
\begin{equation}
\label{eq01}
L(u^\ast)=\sigma^\ast_{ijm,i}(x,\xi,\omega)+\rho\tilde{\omega}^2u^\ast_{jm}(x,\xi,\omega)=-\delta_{jm}\delta(x-\xi), \ \ x,\xi\in \mathbb{R}^3
\end{equation}
 where $\sigma^\ast_{ijm}(x,\xi,\omega) = C_{ijkl}u^\ast_{km,l}(x,\xi,\omega)$.

A fundamental solution is not unique, since it is defined with respect to Eq. (\ref{eq1a}) with a zero right-hand side. The derivation and the form of such a fundamental solution  is given in \cite{ES75}, see also \cite{Do93}, \cite{MS96} and  \cite{ZG98}, Appendix C.

The wave numbers for pressure and shear waves are now defined as $k_1=\frac{\tilde{\omega}}{C_1}=\tilde{\omega}\sqrt{\frac{\rho}{\lambda+2\mu}}$ and $k_2=\frac{\tilde{\omega}}{C_2}=\tilde{\omega}\sqrt{\frac{\rho}{\mu}}$, respectively, where the radial distance between two field points is $r(x,\xi)=\sqrt{(x_1-\xi_1)^2+(x_2-\xi_2)^2+(x_3-\xi_3)^2}$. By defining functions $\tilde{\varphi}_j(r,\omega)= \frac{e^{ik_jr}}{r}, j=1,2$ the fundamental solution of Eq. (\ref{eq01}) is a $3\times3$ matrix-valued function with components:
$$
\tilde{u}(x,\xi,\omega)_{ij}=\frac{1}{4\pi\rho\omega^2}\left[\frac{\tilde{\varphi}_2(x,\xi,\omega)}{r}-\frac{\tilde{\varphi}_1(x,\xi,\omega)}{r}\right]_{,ij}
+\frac{1}{4\pi\mu}\frac{\tilde{\varphi}_2(x,\xi,\omega)}{r}\delta_{ij}.
$$
The aforementioned functions $-\frac{1}{4\pi}\tilde{\varphi}_j(x,\xi,\omega)$ are the fundamental solutions of the Helmholtz equation, see \cite{Vl71}, \S 6.5.
\begin{equation}
\label{eq3}
\Delta \tilde{\varphi}_j(x,\xi,\omega)+k^2_j\tilde{\varphi}_j(x,\xi,\omega)=\delta(x-\xi),
\end{equation}
where $\Delta$ is the Laplace operator in $\mathbb{R}^3$, $\Delta=\frac{\partial}{\partial x_1}+\frac{\partial}{\partial x_2}+\frac{\partial}{\partial x_3}$.

Next, we will find another class of functions $\tilde{\varphi}_j(x,\xi)$ which solve Eq. (\ref{eq3}), but are more convenient for the derivation of Green`s functions. 
\begin{lemma}
\label{lem1}
Let $\beta=\sqrt{\eta^2_1+\eta^2_2-k^2}$, then function
$$
\varphi(x,\xi,\omega)=A\int_{\mathbb{R}^2}\frac{e^{-\beta|x_3-\xi_3|}}{\beta}e^{i[
\eta_1(x_1-\xi_1)+\eta_2(x_2-\xi_2)]}d\eta_1d\eta_2,A= \frac{1}{4\pi},
$$
is a solution of Eq. (\ref{eq3}).
\end{lemma}
\begin{proof}
Function $\varphi(x,\xi,\omega)$ is constructed using the inverse Fourier transform, which is defined in $\mathbb{R}^2$ as follows:
$$
\textit{F}^{-1}(h)(y)=\frac{1}{2\pi}\int_{\mathbb{R}^2}h(\zeta)e^{i(\zeta_1 y_1+\zeta_2 y_2)}d\zeta_1d\zeta_2.
$$
where the  Fourier transform is
$$
\textit{F}(f)(\zeta)=\frac{1}{2\pi}\int_{\mathbb{R}^2}f(y)e^{-i(\zeta_1 y_1+\zeta_2 y_2)}dy_1dy_2.
$$
Calculations give the following results for the derivatives of $\varphi(x,\xi,\omega)$:
$$
\begin{array}{ll}
&\varphi_{,1}(x,\xi)=A\int_{\mathbb{R}^2}i\eta_1\frac{e^{\beta|x_3-\xi_3|}}{\beta}e^{i[
\eta_1(x_1-\xi_1)+\eta_2(x_1-\xi_2)]}d\eta_1d\eta_2;
\\
&\varphi_{,11}(x,\xi)=A\int_{\mathbb{R}^2}(-\eta^2_1)\frac{e^{\beta|x_3-\xi_3|}}{\beta}e^{i[
\eta_1(x_1-\xi_1)+\eta_2(x_1-\xi_2)]}d\eta_1d\eta_2,
\\
&\varphi_{,2}(x,\xi)=A\int_{\mathbb{R}^2}i\eta_2\frac{e^{\beta|x_3-\xi_3|}}{\beta}e^{i[
\eta_1(x_1-\xi_1)+\eta_2(x_1-\xi_2)]}d\eta_1d\eta_2;
\\
&\varphi_{,22}(x,\xi)=A\int_{\mathbb{R}^2}(-\eta^2_2)\frac{e^{\beta|x_3-\xi_3|}}{\beta}e^{i[
\eta_1(x_1-\xi_1)+\eta_2(x_1-\xi_2)]}d\eta_1d\eta_2,
\\
&\varphi_{,3}(x,\xi)=A\int_{\mathbb{R}^2}\beta\frac{e^{\beta|x_3-\xi_3|}}{\beta}\hbox{sgn}(x_3-\xi_3)e^{i[
\eta_1(x_1-\xi_1)+\eta_2(x_1-\xi_2)]}d\eta_1d\eta_2;
\\
&\varphi_{,33}(x,\xi)=A\int_{\mathbb{R}^2}\left[\beta^2\frac{e^{\beta|x_3-\xi_3|}}{\beta}\hbox{sgn}^2(x_3-\xi_3)+\beta\frac{e^{\beta|x_3-\xi_3|}}
{\beta}2\delta(x_3-\xi_3)\right]
\\
&\times e^{i[
\eta_1(x_1-\xi_1)+\eta_2(x_1-\xi_2)]}d\eta_1d\eta_2.
\end{array}
$$
In the above, the  derivative of the signum function is defined as $\hbox{sgn}(x_3-\xi_3)_{,3}=2\delta(x_3-\xi_3)$ and $\hbox{sgn}^2(x_3-\xi_3)=1$.

By reconstituting the Helmholtz operator in Eq. (\ref{eq3}) using the above results, we obtain
$$\begin{array}{ll}
&\Delta\varphi(x,\xi)+k^2\varphi(x,\xi)=A\int_{\mathbb{R}^2}\left[\left(\beta^2-\eta^2_1-\eta^2_2+k^2\right)\frac{e^{\beta|x_3-\xi_3|}}{\beta}
+e^{\beta|x_3-\xi_3|}2\delta(x_3-\xi_3)\right]
\\
&\times e^{i[\eta_1(x_1-\xi_1)+\eta_2(x_1-\xi_2)]}d\eta_1d\eta_2=2A\delta(x-\xi),
\end{array}
$$
Note that the property of the $\delta$ function, $e^{\beta|x_3-\xi_3|}2\delta(x_3-\xi_3)=2\delta(x_3-\xi_3)$, was used, as well as the following property of the inverse Fourier transform:
$$
A\int_{\mathbb{R}^2}\delta(x_3-\xi_3)e^{i[\eta_1(x_1-\xi_1)+\eta_2(x_1-\xi_2)]}d\eta_1d\eta_2=A\delta(x_1-\xi_1)\delta(x_2-\xi_2)\delta(x_3-\xi_3).
$$
By setting constant $A=-\frac{1}{8\pi}$, we finally obtain Eq. (\ref{eq3}).

\end{proof}
\section{Correction term}
\label{sec4}
By using the Lemma \ref{lem1} for a $3\times3$ matrix-valued fundamental solution $u^\ast(x,\xi,\omega)$ with components
$$
u^\ast(x,\xi,\omega)_{ij}=\frac{1}{4\pi\mu}\left\{\left[\frac{1}{k_2^2}\varphi_2(x,\xi,\omega)-\frac{1}{k_2^2}\varphi_1(x,\xi,\omega)\right]_{,ij} +\varphi_2(x,\xi,\omega)\delta_{ij}\right\}
$$
 a correction term $w(x,\xi,\omega)$, $x,\xi \in \mathbb{R}^3_-$ satisfying the following two conditions is derived:
\begin{itemize}
\item[(a)] The $3\times3$ matrix valued function  $w(x,\xi,\omega)$  is a general solution of matrix equation (\ref{eq01}) with a zero right-hand side, i.e., $L(w)=0$;
\item[(b)] The tractions derived from displacements $w (x,\xi,\omega)$ on $x_3=0$ minus the tractions derived from displacements $u^\ast (x,\xi,\omega)$ on $x_3=0$ equilibrate, i.e., condition  Eq. (\ref{eq1b}) for
 $g(x,\xi,\omega)=u^\ast(x,\xi,\omega)+w(x,\xi,\omega)$ holds.
\end{itemize}
Starting with (a): Let $f$ be a $3\times1$ vector-valued function such that
\begin{equation}
\label{eq02}
f(x)=\int_{\mathbb{R}^2}\nu(\eta_1,\eta_2,\beta)e^{\beta(x_3+\xi_3)}e^{i[\eta_1(x_1-\xi_1)+\eta_2(x_2-\xi_2)]}d\eta_1d\eta_2,
\end{equation}
The viscoelasticity operator $M$ in coordinate notations reads as follows:
$$
\begin{array}{ll} &M(\partial_1,\partial_2,\partial_3,\omega)
\\
&=\left(\begin{array}{l}(\lambda+\mu)\partial^2_{x_1^2}+\mu(\partial^2_{x_1^2}+\partial^2_{x_2^2}+\partial^2_{x_3^2})+\rho\omega^2
\quad (\lambda+\mu)\partial^2_{x_1x_2}\quad(\lambda+\mu)\partial^2_{x_1x_3}
\\
(\lambda+\mu)\partial^2_{x_1x_2}\quad(\lambda+\mu)\partial^2_{x_2^2}+\mu(\partial^2_{x_1^2}+\partial^2_{x_2^2}+\partial^2_{x_3^2})+\rho\omega^2
\quad(\lambda+\mu)\partial^2_{x_2x_3}
\\
(\lambda+\mu)\partial^2_{x_1x_3}\quad(\lambda+\mu)\partial^2_{x_2x_3}\quad(\lambda+\mu)\partial^2_{x_3^2}+\mu(\partial^2_{x_1^2}+\partial^2_{x_2^2}+\partial^2_{x_3^2})+\rho\omega^2
\end{array}\right)
\end{array},
$$
where $\partial^2_{x_ix_j}=\frac{\partial^2}{\partial x_i\partial x_j}$. The vector function $f(x)$ is a solution of equation $M(\partial_1,\partial_2,\partial_3,\omega)f(x)=0$ if $\beta$ is an eigenvalue of matrix $M(\eta_1,\eta_2,\beta,\omega)$ and $\nu(\eta_1,\eta_2,\beta)$ is the corresponding eigenvector.
\begin{lemma}
\label{lem2}
The three eigenvalues of matrix $M(\eta_1,\eta_2,\beta,\omega)$ with respect to parameter $\beta$ are $\beta_1=\sqrt{\eta_1^2+\eta_2^2-k_1^2}$, $\beta_2=\beta_3=\sqrt{\eta_1^2+\eta_2^2-k_2^2}$ and the corresponding eigenvectors are:
\begin{equation}
\label{eq03}
\nu^1=\left(-\eta_1, -\eta_2, i\beta_1\right), \quad \nu^2=\left(-\beta_2, 0, i\eta_1\right), \quad \nu^3=\left(-\eta_2,\eta_1, 0\right).
\end{equation}
\end{lemma}
\begin{proof}
By applying operator $M(\partial_1,\partial_2,\partial_3,\omega)$ to $f(x)$ in Eq. (\ref{eq02}) we obtain
$$ M(\partial_1,\partial_2,\partial_3,\omega)f(x)=\int_{\mathbb{R}^2}M(\eta_1,\eta_2,\beta,\omega)\nu(\eta_1,\eta_2,\beta)e^{\beta(x_3+\xi_3)}e^{i[\eta_1(x_1-\xi_1)+\eta_2(x_2-\xi_2)]} d\eta_1d\eta_2, $$ where $$\begin{array}{ll} &M(\eta_1,\eta_2,\beta,\omega)
\\
&=\left(\begin{array}{l}(\lambda+\mu)(-\eta_1^2)+\mu(-\eta_1^2-\eta_2^2+\beta^2)+\rho\omega^2
\quad (\lambda+\mu)(-\eta_1\eta_2)\quad(\lambda+\mu)i\eta_1\beta
\\
(\lambda+\mu)(-\eta_1\eta_2)\quad(\lambda+\mu)(-\eta_2^2)+\mu(-\eta_1^2-\eta_2^2+\beta^2)+\rho\omega^2
\quad(\lambda+\mu)i\eta_2\beta
\\
(\lambda+\mu)i\eta_1\beta\quad(\lambda+\mu)i\eta_2\beta\quad(\lambda+\mu)\beta^2+\mu(-\eta_1^2-\eta_2^2+\beta^2)+\rho\omega^2.
\end{array}\right)
\end{array}.
$$
Matrix equation $M(\partial_1,\partial_2,\partial_3,\omega)f(x)=0$ has a nontrivial solution if and only if
\begin{equation}
\label{eq4}
\hbox{det}M(\eta_1,\eta_2,\beta,\omega)=0.
\end{equation}
By evaluating this determinant from Eq. (\ref{eq4}) we have
$$
\hbox{det}M(\eta_1,\eta_2,\beta,\omega)=(\beta^2-\eta_1^2-\eta_2^2+k_1^2)(\beta^2-\eta_1^2-\eta_2^2+k_2^2)^2=0
$$
and the eigenvalues of $M$ with respect to $\beta$ are
$$
\beta_1=\sqrt{\eta_1^2+\eta_2^2-k_1^2}, \beta_2=\beta_3=\sqrt{\eta_1^2+\eta_2^2-k_2^2}.
$$
In order to recover the eigenvectors, we must solve equations $M(\eta_1,\eta_2,\beta_j,\omega)\nu^j=0$ for $j=1,2,3$.

Now check that Eq. (\ref{eq03}) does give the correct eigenvectors. For vector $\nu^1$ the expansion in rows gives
$$
\begin{array}{ll}
&\left[M(\eta_1,\eta_2,\beta_1,\omega)\nu^1\right]_1
\\
&=\left[(\lambda+\mu)(-\eta_1^2)+\mu(-\eta_1^2-\eta_2^2+\beta_1^2)+\rho\tilde{\omega}^2\right](-\eta_1)+
(\lambda+\mu)(-\eta_1)(-\eta_2)+(\lambda+\mu)i\eta_1\beta_1i\beta_1
\\
&=-\eta_1\left[(\lambda+\mu)(-\eta_1^2)+\mu\left(-\frac{\rho\omega^2}{\lambda+2\mu}+\rho\tilde{\omega}^2\right)+(\lambda+2\mu)(-\eta_2^2+\beta_1^2)\right]
\\
&=-\eta_1\left[-\frac{(\lambda+\mu)\rho\tilde{\omega}^2}{\lambda+2\mu}-\frac{\mu\rho\tilde{\omega}^2}{\lambda+2\mu}+\rho\tilde{\omega}^2\right]=0
\\
&\left[M(\eta_1,\eta_2,\beta_1,\omega)\nu^1\right]_2
\\
&=(\lambda+\mu)(-\eta_1\eta_2)(-\eta_1)+\left[(\lambda+\mu)(-\eta_2^2)+\mu(-\eta_1^2-\eta_2^2+\beta_1^2)+\rho\tilde{\omega}^2\right](-\eta_2)+
(\lambda+\mu)i\eta_2\beta_1i\beta_1
\\
&=-\eta_2\left[(\lambda+\mu)(-\eta_1^2)+(\lambda+\mu)(-\eta_2^2)+\mu(-\eta_1^2-\eta_2^2+\beta_1^2)+\rho\tilde{\omega}^2+(\lambda+\mu)\beta_1^2\right]
\\
&=-\eta_2\left[(\lambda+\mu)(-\eta_1^2-\eta_2^2)+(\lambda+\mu)\beta_1^2+\rho\tilde{\omega}^2\right]=0
\\
&\left[M(\eta_1,\eta_2,\beta_1,\omega)\nu^1\right]_3
\\
&=(\lambda+\mu)(i\eta_1\beta_1)(-i\eta_1)+(\lambda+\mu)(i\eta_2\beta_1)(-i\eta_2)+\left[(\lambda+\mu)\beta_1^2+\mu(-\eta_1^2-\eta_2^2)
+\rho\tilde{\omega}^2\right]i\beta_1
\\
&=i\beta_1\left[(\lambda+\mu)(-\eta_1^2)+(\lambda+\mu)(-\eta_2^2)+(\lambda+\mu)\beta_1^2)+\mu(-\eta_1^2-\eta_2^2+\beta_1^2)+\rho\tilde{\omega}^2\right]
\\
&=i\beta_1\left[(\lambda+2\mu)(-\eta_1^2-\eta_2^2+\beta_1^2)+\rho\omega^2\right]=0
\end{array}
$$
Similarly for vector $\nu^2$ we have
$$
\begin{array}{ll}
&\left[M(\eta_1,\eta_2,\beta_2,\omega)\nu^2\right]_1
\\
&=\left[(\lambda+\mu)(-\eta_1^2)+\mu(-\eta_1^2-\eta_2^2+\beta_2^2)+\rho\omega^2\right](-\beta_2)+
(\lambda+\mu)i\eta_1i\eta_1
\\
&=(\lambda+\mu)(-\eta_1^2)(-\beta_2)+(\lambda+\mu)(-\eta_1^2)\beta_2=0
\\
&\left[M(\eta_1,\eta_2,\beta_2,\omega)\nu^2\right]_2
\\
&=(\lambda+\mu)(-\eta_1\eta_2)(-\beta_2)+(\lambda+\mu)i\eta_2i\eta_1=0
\\
&\left[M(\eta_1,\eta_2,\beta_2,\omega)\nu^2\right]_3
\\
&=(\lambda+\mu)(i\eta_1\beta_2)(-\beta_2)+\left[(\lambda+\mu)\beta_2^2+\mu(-\eta_1^2-\eta_2^2+\beta_2^2)
+\rho\tilde{\omega}^2\right]i\eta_1
\\
&=(\lambda+\mu)(-\beta^2_2)i\eta_1+(\lambda+\mu)\beta_2^2i\eta_1=0
\end{array}
$$
Finally, for vector $\nu^3$ we have
$$
\begin{array}{ll}
&\left[M(\eta_1,\eta_2,\beta_3,\omega)\nu^3\right]_1
\\
&=\left[(\lambda+\mu)(-\eta_1^2)+\mu(-\eta_1^2-\eta_2^2+\beta_3^2)+\rho\tilde{\omega}^2\right](-\eta_2)+
(\lambda+\mu)(-\eta_1^2\eta_2)
\\
&=(\lambda+\mu)(\eta_1^2\eta_2)+(\lambda+\mu)(-\eta_1^2\eta_2)=0
\\
&\left[M(\eta_1,\eta_2,\beta_3,\omega)\nu^3\right]_2
\\
&=(\lambda+\mu)(-\eta_1\eta_2)(-\eta_2)+\left[(\lambda+\mu)(-\eta_2^2+\mu(-\eta_1^2-\eta_2^2+\beta_3^2)+\rho\omega^2\right]\eta_1
\\
&=(\lambda+\mu)\eta_2^2\eta_1+(\lambda+\mu)(-\eta_2^2)\eta_1=0
\\
&\left[M(\eta_1,\eta_2,\beta_3,\omega)\nu^3\right]_3
\\
&=(\lambda+\mu)(i\eta_1\beta_3)(-\eta_2)+(\lambda+\mu)i\beta_3\eta_2\eta_1
\\
&=(\lambda+\mu)(-i\beta_3\eta_2)\eta_1+(\lambda+\mu)i\beta_3\eta_2\eta_1=0
\end{array}
$$

\end{proof}
Using Lemma \ref{lem2} allows us to form the general solution of Eq. (\ref{eq1a}) with a zero right-hand side as a $3\times3$ matrix valued function $w(x,\xi,\omega)$ comprising vectors $w=(w_{j1},w_{j2},w_{j3})$. Here
\begin{equation}
\label{eq5}
w_{jk}=\int_{\mathbb{R}^2}C^m_k\nu^m_je^{\beta_m(x_3+\xi_3)}e^{i[\eta_1(x_1-\xi_1)+\eta_2(x_2-\xi_2)]}d\eta_1d\eta_2, \quad  j=1,2,3,
\end{equation}
where the nine component $C^m_k$ depend on parameters $\eta_1,\eta_2,\beta_m,\omega$, but not on coordinates $x_1,x_2$. We can find the $C^m_k$ from the boundary conditions
\begin{equation}
\label{eq6}
T^w_{pk}=-T^{u^\ast}_{pk} \quad  \hbox{on} \quad x_3=0.
\end{equation}

By using the definition of the $3\times3$ traction matrix along $x_3=0$ we have
\begin{equation}
\label{eq7}
T^w_{pk}|_{x_3=0}=C_{p3jm}w_{jk,m}|_{x_3=0},
\end{equation}
Finally, the following expressions are obtained for the tractions after applying some simplifications in the calculation process
$$
\begin{array}{ll}
&T^w_{1k}|_{x_3=0}
=\int_{\mathbb{R}^2}\mu\left[C^1_k(-2\eta_1\beta_1)\frac{e^{\beta_1\xi_3}}{\beta_1}+C^2_k(-\beta_2^2-\eta_1^2)\frac{e^{\beta_2\xi_3}}{\beta_2}+C^3_k (-\eta_2\beta_3)\frac{e^{\beta_3\xi_3}}{\beta_3}\right]
\\
&\times e^{i[\eta_1(x_1-\xi_1)+\eta_2(x_2-\xi_2)]}d\eta_1d\eta_2;
\\
&T^w_{2k}|_{x_3=0}
=\int_{\mathbb{R}^2}\mu\left[C^1_k(-\eta_2\beta_1)\frac{e^{\beta_1\xi_3}}{\beta_1}+C^2_k(-\eta_1\eta_2)\frac{e^{\beta_2\xi_3}}{\beta_2}+C^3_k (\eta_1\beta_3)\frac{e^{\beta_3\xi_3}}{\beta_3}\right]
\\
&\times e^{i[\eta_1(x_1-\xi_1)+\eta_2(x_2-\xi_2)]}d\eta_1d\eta_2;
\\
&T^w_{3k}|_{x_3=0}
=\int_{\mathbb{R}^2}C^1_k\left[\lambda i(-\eta_1^2-\eta_2^2)+(\lambda+2\mu)i\beta_1^2\right]\frac{e^{\beta_1\xi_3}}{\beta_1}
\\
&+C^2_k\left[\lambda(-i\eta_1\beta_2)+(\lambda+2\mu)i\eta_1\beta_2 \frac{e^{\beta_2\xi_3}}{\beta_2}\right]
 e^{i[\eta_1(x_1-\xi_1)+\eta_2(x_2-\xi_2)]}d\eta_1d\eta_2.
\end{array}
$$

Components of a fundamental solution $u^\ast(x,\xi,\omega)$ are
$$
 u^\ast_{pq}=\frac{1}{4\pi\mu}\int_{\mathbb{R}^2}D_{pq}e^{i[\eta_1(x_1-\xi_1)+\eta_2(x_2-\xi_2)]}d\eta_1d\eta_2,
 $$
 where
 $$
 D_{pq}=D^2_{pq}\frac{1}{k_2^2}\frac{e^{-\beta_2|x_3-\xi_3|}}{\beta_2}-D^1_{pq}\frac{1}{k_2^2}\frac{e^{-\beta_1|x_3-\xi_3|}}{\beta_1},
 $$
 and $D^1_{pq}$, $D^2_{pq}$ are the components of the following $3\times3$ matrices:
 $$
  D^1=\left(\begin{array}{lll} -\eta_1^2 & -\eta_1\eta_2 & i\eta_1\beta_1\hbox{sgn}(x_3-\xi_3)
\\
-\eta_1\eta_2 & -\eta_2^2 & i\eta_2\beta_1\hbox{sgn}(x_3-\xi_3)
\\
i\eta_1\beta_1\hbox{sgn}(x_3-\xi_3) & i\eta_2\beta_1\hbox{sgn}(x_3-\xi_3) & \beta_1^2-2\beta_1\delta(x_3-\xi_3)
\end{array}\right)
$$
$$
D^2=\left(\begin{array}{lll} -\eta_1^2+k^2_2 & -\eta_1\eta_2 & i\eta_1\beta_2\hbox{sgn}(x_3-\xi_3)
\\
-\eta_1\eta_2 & -\eta_2^2+k^2_2  & i\eta_2\beta_2\hbox{sgn}(x_3-\xi_3)
\\
i\eta_1\beta_2\hbox{sgn}(x_3-\xi_3) & i\eta_2\beta_2\hbox{sgn}(x_3-\xi_3) & \beta_2^2+k^2_2 -2\beta_2\delta(x_3-\xi_3)
\end{array}\right)
$$
In the above using the properties of $\delta$ function: $\delta(x_3-\xi_3)e^{-\beta_j|x_3-\xi_3|}=\delta(x_3-\xi_3)$ for $D_{33}$ we obtain
$$
D_{33}=\frac{1}{k_2^2}\left[D_{33}^2\frac{e^{-\beta_2|x_3-\xi_3|}}{\beta_2}-D_{33}^1\frac{e^{-\beta_1|x_3-\xi_3|}}{\beta_1}\right]
=\frac{1}{k_2^2}\left[(\beta^2_2+k_2^2)\frac{e^{-\beta_2|x_3-\xi_3|}}{\beta_2}-\beta^2_1\frac{e^{-\beta_1|x_3-\xi_3|}}{\beta_1}\right],
$$
and terms with $\delta$ function are simplified.

The resulting tractions generate by this fundamental solution along the boundary $x_3=0$ are
$$
 T^{u^\ast}_{jk}|_{x_3=0}=C_{j3mn}u^\ast_{mk,n}|_{x_3=0}.
 $$
Their  form   in coordinate notations is as follows:
\begin{equation}
\label{eq9}
\begin{array}{ll}
&T^{u^\ast}_{1k}|_{x_3=0}
=\int_{\mathbb{R}^2}\frac{1}{4\pi k_2^2}\left[\left(D^1_{1k}\beta_1+D^1_{3k}i\eta_1\right)\frac{e^{\beta_1\xi_3}}{\beta_1}-
\left(D^2_{1k}\beta_2+D^2_{3k}i\eta_1\right)\frac{e^{\beta_2\xi_3}}{\beta_2}\right]
\\
&\times e^{i[\eta_1(x_1-\xi_1)+\eta_2(x_2-\xi_2)]}d\eta_1d\eta_2=\int_{\mathbb{R}^2}Q_{1k}e^{i[\eta_1(x_1-\xi_1)+\eta_2(x_2-\xi_2)]}d\eta_1d\eta_2,
\\
&T^{u^\ast}_{2k}|_{x_3=0}
=\int_{\mathbb{R}^2}\frac{1}{4\pi k_2^2}\left[\left(D^1_{2k}\beta_1+D^1_{3k}i\eta_2\right)\frac{e^{\beta_1\xi_3}}{\beta_1}-
\left(D^2_{2k}\beta_2+D^2_{3k}i\eta_2\right)\frac{e^{\beta_2\xi_3}}{\beta_2}\right]
\\
&\times e^{i[\eta_1(x_1-\xi_1)+\eta_2(x_2-\xi_2)]}d\eta_1d\eta_2=\int_{\mathbb{R}^2}Q_{2k}e^{i[\eta_1(x_1-\xi_1)+\eta_2(x_2-\xi_2)]}d\eta_1d\eta_2,
\\
&T^{u^\ast}_{3k}|_{x_3=0}
=\int_{\mathbb{R}^2}\frac{1}{4\pi k_2^2}\left\{(\lambda+2\mu)D^1_{3k}\beta_1+\lambda\left(D^1_{1k}i\eta_1+D^1_{2k}i\eta_2\right)\frac{e^{\beta_1\xi_3}}{\beta_1}\right.
\\
&\left.-\left[(\lambda+2\mu)D^2_{3k}\beta_2+\lambda\left(D^2_{1k}i\eta_1+D^2_{2k}i\eta_2\right)\frac{e^{\beta_2\xi_3}}{\beta_2}\right]\right\}.
\\
&\times e^{i[\eta_1(x_1-\xi_1)+\eta_2(x_2-\xi_2)]}d\eta_1d\eta_2=\int_{\mathbb{R}^2}Q_{3k}e^{i[\eta_1(x_1-\xi_1)+\eta_2(x_2-\xi_2)]}d\eta_1d\eta_2.
\end{array}
\end{equation}
Note that we have used the properties of the signum function $\hbox{sgn}(x_3-\xi_3)|_{x_3=0}=\hbox{sgn}(-\xi_3)=1$, since $\xi_3<0$. Also, we used the properties of the $\delta$ function to obtain that $\delta(x_3-\xi_3)|_{x_3=0}=0$ since $\xi_3<0$.

The expressions for $Q_k$ using $D^1, D^2$ and Eq. (\ref{eq9}) are:
\begin{itemize}
\item[(1k):] $Q_{1k}=\frac{\mu}{4\pi k^2_2}\left[\left(D_{1k}^1\beta_1+D_{3k}^1i\eta_1\right)\frac{e^{\beta_1\xi_3}}{\beta_1}-\left(D_{1k}^2\beta_2+D_{3k}^2i\eta_1\right)\frac{e^{\beta_2\xi_3}}{\beta_2}\right]$
\begin{itemize}
\item[(11):] $Q_{11}=-\frac{\mu}{4\pi k^2_2}\left[\left(2\eta_1^2\beta_1\right)\frac{e^{\beta_1\xi_3}}{\beta_1}+\left(-2\eta_1^2\beta_2+k_2^2\beta_2\right)\frac{e^{\beta_2\xi_3}}{\beta_2}\right]$
\item[(12):] $Q_{12}=-\frac{\mu}{4\pi k^2_2}\left[\left(2\eta_1\eta_2\beta_1\right)\frac{e^{\beta_1\xi_3}}{\beta_1}+\left(-2\eta_1\eta_2\beta_2+k_2^2\beta_2\right)\frac{e^{\beta_2\xi_3}}{\beta_2}\right]$
\item[(13):] $Q_{13}=-\frac{\mu}{4\pi k^2_2}\left[\left(2i\eta_1\beta_1^2\right)\frac{e^{\beta_1\xi_3}}{\beta_1}-\left(i\eta_1(2\beta_2^2+k^2_2)\right)\frac{e^{\beta_2\xi_3}}{\beta_2}\right]$
\end{itemize}
\item[(2k):]$Q_{2k}=\frac{\mu}{4\pi k^2_2}\left[\left(D_{2k}^1\beta_1+D_{3k}^1i\eta_2\right)\frac{e^{\beta_1\xi_3}}{\beta_1}-\left(D_{2k}^2\beta_2+D_{3k}^2i\eta_2\right)\frac{e^{\beta_2\xi_3}}{\beta_2}\right]$
\begin{itemize}
\item[(21):]$Q_{21}=-\frac{\mu}{4\pi k^2_2}\left[\left(2\eta_1\eta_2\beta_1\right)\frac{e^{\beta_1\xi_3}}{\beta_1}+\left(-2\eta_1\eta_2\beta_2+k_2^2\beta_2\right)\frac{e^{\beta_2\xi_3}}{\beta_2}\right]$
\item[(22):]$Q_{22}=-\frac{\mu}{4\pi k^2_2}\left[\left(\eta_2^2\beta_1\right)\frac{e^{\beta_1\xi_3}}{\beta_1}+\left(-2\eta_2^2\beta_2+k_2^2\beta_2\right)\frac{e^{\beta_2\xi_3}}{\beta_2}\right]$
\item[(23):]$Q_{23}=\frac{\mu}{4\pi k^2_2}\left[\left(2i\eta_2\beta^2_1\right)\frac{e^{\beta_1\xi_3}}{\beta_1}-\left(2i\eta_2\beta_2^2+i\eta_2k_2^2)\right)\frac{e^{\beta_2\xi_3}}{\beta_2}\right]$
\end{itemize}
\item[(3k):]$\begin{array}{l}Q_{3k}=\frac{1}{4\pi k^2_2}\left\{\left[(\lambda+2\mu)D^1_{3k}\beta_1+\lambda\left(D_{1k}^1i\eta_1+D_{2k}^1i\eta_2\right)\right]\frac{e^{\beta_1\xi_3}}{\beta_1}\right.
    \\
    \left.-
    \left[(\lambda+2\mu)D^2_{3k}\beta_2+\lambda\left(D_{1k}^2i\eta_1+D_{2k}^2i\eta_2\right)\right]\frac{e^{\beta_2\xi_3}}{\beta_2}\right\}\end{array}$
\begin{itemize}
\item[(31):]$\begin{array}{l}Q_{31}=\frac{1}{4\pi k^2_2}\left\{\left[(\lambda+2\mu)i\eta_1\beta^2_1+\lambda\left(-i\eta_1^3-i\eta_1\eta^2_2\right)\right]\frac{e^{\beta_1\xi_3}}{\beta_1}\right.
    \\
    \left.-
    \left[(\lambda+2\mu)i\eta_1\beta^2_2+\lambda\left((-\eta^2_1+k_2^2)i\eta_1-i\eta_1\eta^2_2\right)\right]\frac{e^{\beta_2\xi_3}}{\beta_2}\right\}\end{array}$
\item[(32):]$\begin{array}{l}Q_{32}=\frac{1}{4\pi k^2_2}\left\{\left[(\lambda+2\mu)i\eta_2\beta^2_1+\lambda\left(-i\eta_1^2\eta_2-i\eta^3_2\right)\right]\frac{e^{\beta_1\xi_3}}{\beta_1}\right.
    \\
    \left.-
    \left[(\lambda+2\mu)i\eta_2\beta^2_2+\lambda\left((-i\eta^2_1\eta_2+i\eta_2(-\eta_1^2+k_2^2)\right)\right]\frac{e^{\beta_2\xi_3}}{\beta_2}\right\}\end{array}$
\item[(33):]$\begin{array}{l}Q_{33}=\frac{1}{4\pi k^2_2}\left\{\left[(\lambda+2\mu)\beta^3_1+\lambda\left(-\beta_1(\eta_1^2+\eta_2^2)\right)\right]\frac{e^{\beta_1\xi_3}}{\beta_1}\right.
    \\
    \left.-
    \left[(\lambda+2\mu)\beta_2(\beta^2_2+k_2^2)+\lambda\left(-\beta_2(\eta_1^2+\eta_2^2)\right)\right]\frac{e^{\beta_2\xi_3}}{\beta_2}\right\}\end{array}$
\end{itemize}
\end{itemize}

The matrix of the coefficients $C^m_k$ is $N=\left\{N_{ij}\right\}$, $i,j=1,2,3$:
$$
 N=\left(\begin{array}{lll}
\mu(-2\eta_1\beta_1)\frac{e^{\beta_1\xi_3}}{\beta_1} &\mu(-\beta_2^2-\eta_1^2)\frac{e^{\beta_2\xi_3}}{\beta_2}&
\mu(-\eta_2\beta_3)\frac{e^{\beta_3\xi_3}}{\beta_3}
\\
\mu(-2\eta_2\beta_1)\frac{e^{\beta_1\xi_3}}{\beta_1} & \mu(-\eta_1\eta_2)\frac{e^{\beta_2\xi_3}}{\beta_2}&
\mu(\eta_1\beta_3) \frac{e^{\beta_3\xi_3}}{\beta_3}
\\
\begin{array}{l}\left[\lambda(-i\eta^2_1-i\eta_2^2)\right.
\\ \left.+i(\lambda+2\mu)\beta_1^2\right]\frac{e^{\beta_1\xi_3}}{\beta_1}\end{array}
& \begin{array}{l}\left[\lambda(-i\eta_1\beta_2)\right.
\\ \left.+(\lambda+2\mu)i\eta_1\beta_2\right]\frac{e^{\beta_2\xi_3}}{\beta_2}\end{array} & 0
\end{array}\right)
$$
If we denote the determinant of the above matrix as $\Delta=\det N$, we get upon simplification that
\begin{equation}
\label{eq90}
\begin{array}{l}
\Delta=\frac{e^{(\beta_1+2\beta_2)\xi_3}}{\beta_1\beta_2^2}\left\{i\mu^2\eta_1\beta_2\{4\mu (\eta_1 ^2+\eta_2^2)\beta_1\beta_2\right.
\\
+\left.[\lambda(\eta_1^2+\eta_2^2)-(\lambda+2\mu)\beta_1^2][2(\eta_1^2+\eta_2^2)-k_2^2]\}\right\}.
\end{array}
\end{equation}
We now state the following hypothesis (H):

If $p\neq q$ and $\alpha\neq0$, then the only root of $\Delta$ is $\eta_1=0$.
\begin{theorem}
\label{th1}
Let hypothesis (H) hold true. Then, there exists a $3\times3$ matrix-valued function $w$ such that $g(x,\xi,\omega)=u^\ast(x,\xi,\omega)+w(x,\xi,\omega)$ is a solution of Eq. (\ref{eq1a}) and satisfies boundary conditions Eq.  (\ref{eq6}), i.e., Eq. (\ref{eq1b}) holds.
\end{theorem}
\begin{proof}
We can recover $w(x,\xi,\omega)$ by solving three linear systems of equations for vectors $C_k=(C^1_k,C^2_k,C^3_k)$ with a right-hand-side comprising vectors $Q_k=(Q_{1k},Q_{2k},Q_{3k})$
\begin{equation}
\label{eq10}
NC_k=-Q_k, \quad k=1,2,3.
\end{equation}

Note that if (H) holds true, the zeros of $\Delta$ (which is in the denominator of all solutions of the system of Eq. (\ref{eq10}) and under the integral sign in $\mathbb{R}^2$) do not affect on the solvability of Eq. (\ref{eq10}).

The solution of Eq. (\ref{eq10}) for $C^m_k$ is expressed using Kramer's rule as
$$
C^m_k=\frac{\Delta_{mk}}{\Delta},
$$
where $\Delta$ is given by Eq. (\ref{eq90}) and $\Delta_{mk}$ derives from $\Delta$ replacing column $(m)$ with column $(Q_{m1},Q_{m2},Q_{m3})^t$, where superscript $t$ denote transposition.

Finally, the expressions for $\Delta_{mk}$ are as follows:
$$
\begin{array}{ll}
 &\Delta_{11}=\left|\begin{array}{lll} -Q_{11}&N_{12}&N_{13}
\\-Q_{12}&N_{22}&N_{23}
\\-Q_{13}&N_{32}&N_{33}\end{array}\right|, \quad \Delta_{12}=\left|\begin{array}{lll} N_{11}&-Q_{11}&N_{13}
\\N_{21}&-Q_{12}&N_{23}
\\N_{31}&-Q_{13}&N_{33}\end{array}\right|, \quad \Delta_{13}=\left|\begin{array}{lll} N_{11}&N_{12}&-Q_{11}
\\N_{21}&N_{22}&-Q_{12}
\\N_{31}&N_{32}&-Q_{13}\end{array}\right|,
\\[2pt]
\\
&\Delta_{21}=\left|\begin{array}{lll} -Q_{21}&N_{12}&N_{13}
\\-Q_{22}&N_{22}&N_{23}
\\-Q_{23}&N_{32}&N_{33}\end{array}\right|, \quad \Delta_{22}=\left|\begin{array}{lll} N_{11}&-Q_{21}&N_{13}
\\N_{21}&-Q_{22}&N_{23}
\\N_{31}&-Q_{23}&N_{33}\end{array}\right|, \quad \Delta_{23}=\left|\begin{array}{lll} N_{11}&N_{12}&-Q_{21}
\\N_{21}&N_{22}&-Q_{22}
\\N_{31}&N_{32}&-Q_{23}\end{array}\right|,
\\[2pt]
\\
&\Delta_{31}=\left|\begin{array}{lll} -Q_{31}&N_{12}&N_{13}
\\-Q_{32}&N_{22}&N_{23}
\\-Q_{33}&N_{32}&N_{33}\end{array}\right|, \quad \Delta_{32}=\left|\begin{array}{lll} N_{11}&-Q_{31}&N_{13}
\\N_{21}&-Q_{32}&N_{23}
\\N_{31}&-Q_{33}&N_{33}\end{array}\right|, \quad \Delta_{33}=\left|\begin{array}{lll} N_{11}&N_{12}&-Q_{31}
\\N_{21}&N_{22}&-Q_{32}
\\N_{31}&N_{32}&-Q_{33}\end{array}\right|.
\end{array}
$$
The expressions of $\Delta_{mk}$ after simplifications are:
\begin{itemize}
\item[]
$
\begin{array}{l}
\Delta_{11}=\frac{i\mu^3e^{2\beta_2\xi_3}\eta_1}{4\pi k_2^2\beta_2^2}\left\{-2e^{\beta_1\xi_3}\beta_2\eta_1\left[2\beta_2(\eta_1^2+\eta_2^2)+\beta_1(\beta_2^2+\eta_1^2+\eta_2^2)\right]\right.
\\[1pt]
\\
\left.+e^{\beta_2\xi_3}\left[2\beta_2^2\eta_1(\beta_2^2+3(\eta_1^2+\eta_2^2))+k_2^2(-\beta_2^2(\eta_1+2\eta_2)+\eta_1(\eta_1^2+\eta_2^2))\right]\right\}.
\end{array}
$
\item[]
$
\begin{array}{l}
\Delta_{12}=\frac{1}{4\pi k_2^2\beta_1\beta_2}\left\{ie^{(\beta_1+\beta_2)\xi_3}\mu^2\left[-2e^{\beta_1\xi_3}\beta_2\eta_1(\eta_1^2+\eta_2^2)((\eta_1^2+\eta_2^2)\lambda-\beta_1^2(\lambda+4\mu))\right.\right.
\\[1pt]
\\
\left.\left.-e^{\beta_2\xi_3}\left[2\beta_2\eta_1(\eta_1^2+\eta_2^2)((\eta_1^2+\eta_2^2)\lambda+2\beta_1\beta_2\mu+\beta1^2(\lambda+2\mu))\right.\right.\right.
\\[1pt]
\\
\left.\left.+k_2^2(\beta_2(\eta_1^3+\eta_1^2\eta_2+\eta_1\eta_2^2+\eta_2^3)\lambda+2\beta_1\eta_1(\eta_1^2+\eta_2^2)\mu-\beta_1^2\beta_2(\eta_1+\eta_2)(\lambda+2\mu))]\right]\right\}.
\end{array}
$
\item[]
$
\begin{array}{l}
\Delta_{13}=-\frac{1}{4\pi k_2^2\beta_1\beta_2}\left\{ie^{(\beta_1+\beta_2)\xi_3}\mu^2\left[-2e^{\beta_1\xi_3}\beta^2_2\eta_1\eta_2((\eta_1^2+\eta_2^2)\lambda-\beta_1^2(\lambda+4\mu))\right]\right.
\\[1pt]
\\
\left.-e^{\beta_2\xi_3}\left[2\beta^2_2\eta_1\eta_2(-(\eta_1^2+\eta_2^2)\lambda+2\beta_1\beta_2\mu+\beta_1^2(\lambda+2\mu))+k_2^2((\beta^2_2+\eta_1(\eta_1-\eta_2))\right.\right.
\\[1pt]
\\
\left.\left.+(\eta_1^2+\eta_2^2)\lambda+2\beta_1\beta_2\eta_1(2\eta_1-\eta_2)\mu
-\beta_1^2(\beta_2^2+\eta_1(\eta_1-\eta_2))(\lambda+2\mu))\right]\right\}.
\end{array}
$
\item[]
$
\begin{array}{l}
\Delta_{21}=-\frac{i\mu^3e^{2\beta_2\xi_3}\eta_1}{4\pi k_2^2\beta_2^2}\left\{-2e^{\beta_1\xi_3}\beta_2\eta_2\left[\beta_1(\beta_2^2+\eta_1^2+\eta_2^2)-\beta_2(2\eta_1^2+\eta_2^2)\right]\right.
\\[1pt]
\\
\left.+e^{\beta_2\xi_3}\left[2\beta_2^2\eta_2(\beta_2^2-\eta_1^2-\eta_2^2)+k_2^2(2\beta_2^2\eta_1+3\beta_2^2\eta_2+\eta_1^2\eta_2+\eta_2^3)\right]\right\}.
\end{array}
$
\item[]
$
\begin{array}{l}
\Delta_{22}=\frac{1}{4\pi k_2^2\beta_1\beta_2}\left\{ie^{(\beta_1+\beta_2)\xi_3}\mu^2\left[-e^{\beta_1\xi_3}\beta_2\eta_2((2\eta_1^4+3\eta_1^2\eta_2^2)\lambda
-\beta_1^2(2\eta_1^2\lambda+\eta_2^2(\lambda-2\mu))\right.\right.
\\[1pt]
\\
\left.\left.+e^{\beta_2\xi_3}\left[2\beta_2\eta_2(\eta_1^2+\eta_2^2)((\eta_1^2+\eta_2^2)\lambda+2\beta_1\beta_2\mu-\beta^2_(\lambda+2\mu))\right.\right.\right.
\\[1pt]
\\
\left.\left.+k_2^2(-\beta_2(\eta_1^3+\eta_1^2\eta_2+\eta_1\eta_2^2+\eta_2^3)\lambda+2\beta_1\eta_2(\eta_1^2+\eta_2^2)\mu+\beta_1^2\beta_2(\eta_1+\eta_2)(\lambda+2\mu))]\right]\right\}.
\end{array}
$
\item[]
$
\begin{array}{l}
\Delta_{23}=-\frac{1}{4\pi k_2^2\beta_1\beta_2}\left\{ie^{(\beta_1+\beta_2)\xi_3}\mu^2\left[-e^{\beta_1\xi_3}\eta_2^2(-(\beta_2^2-\eta_1^2)(\eta_1^2+\eta_2^2)\lambda+4\beta_1\beta_2\eta_1^2\mu\right.\right.
\\[1pt]
\\
\left.\left.+\beta_1^2(\beta_2^2(\lambda-2\mu)-\eta_1^2(\lambda+2\mu))\right.\right.
\\[1pt]
\\
\left.\left.+e^{\beta_2\xi_3}\left[2\beta_2^2\eta_2^2((\eta_1^2+\eta_2^2)\lambda+2\beta_1\beta_2\mu-\beta_1^2(\lambda+2\mu))+k_2^2[-(\beta_2^2+\eta_1(\eta_1-\eta_2))
(\eta_1^2+\eta_2^2)\lambda\right.\right.\right.
\\[1pt]
\\
\left.\left.+2\beta_1\beta_2\mu(-2\eta_1^2+2\eta_1\eta_2+\eta_2^2)\mu+\beta_1^2(\beta_2^2+\eta_1(\eta_1-\eta_2))(\lambda+2\mu))]\right]\right\}.
\end{array}
$
\item[]
$
\begin{array}{l}
\Delta_{31}=-\frac{1}{4\pi k_2^2\beta_1\beta_2}\left\{e^{2\beta_2\xi_3}\eta_1\mu^2\left[-e^{\beta_1\xi_3}(-2\beta_2(\eta_1^2+\eta^2_1)\right.\right.
\\[1pt]
\\
\left.\left.+\beta_1(\beta_2^2+\eta_1^2+\eta_2^2))(-(\eta_1^2+\eta_2^2)\lambda+\beta_1^2(\lambda+2\mu))\right]\right.
\\[1pt]
\\
\left.+e^{\beta_2\xi_3}\beta_1\left[(\eta_1^4+4\eta_1^2\eta_2^2-\eta_2^4)\lambda-2\beta_2^2(\eta_1^2+\eta_2^2)(\lambda+2\mu)\right.\right.
\\[1pt]
\\
\left.\left.+\beta_2^4(\lambda+2\mu)+k_2^2[-(\eta_1^2+\eta_2^2)(\lambda-2\mu)+\beta_2^2(\lambda+2\mu)]\right]\right\}
\end{array}
$
\item[]
$
\begin{array}{l}
\Delta_{32}=\frac{1}{4\pi k_2^2}\left\{e^{(\beta_1+\beta_2)\xi_3}\mu\left[2(\eta_1^2+\eta^2_1)\mu\right.\right.
\\[1pt]
\\
\left.\left.e^{\beta_1\xi_3}((\eta_1^2+\eta_2^2)\lambda-\beta_1^2(\lambda+2\mu))+e^{\beta_2\xi_3}(-(\eta_1^2+\eta^2_1)\lambda+(k_2^2+\beta_2^2)(\lambda+2\mu))\right]\right.
\\[1pt]
\\
\left.\frac{1}{\beta_1^2\beta_2}[-(\eta_1^2+\eta_2^2)\lambda+\beta_1^2(\lambda+2\mu)]
\left[e^{\beta_1\xi_3}\beta_2(\eta_1^2+\eta_2^2)((\eta_1^2+\eta_2^2)\lambda-\beta_1^2(\lambda+2\mu))\right]\right.
\\[1pt]
\\
\left.+e^{\beta_2\xi_3}\beta_1[k_2^2(\eta_1^2+\eta_2^2)\lambda-\eta_1^2(\eta_1^2+3\eta_2^2)\lambda+\beta_2^2(\eta_1^2+\eta_2^2)(\lambda+2\mu)]\right\}
\end{array}
$
\item[]
$
\begin{array}{l}
\Delta_{33}=\frac{1}{4\pi k_2^2\beta_1^2\beta_2^2}\left\{e^{(\beta_1+\beta_2)\xi_3}\eta_2\mu\left[e^{\beta_1\xi_3}\beta_2^3(\beta_1^2-\eta_1^2-\eta_2^2)
\lambda(-(\eta_1^2+\eta_2^2)\lambda+\beta_1^2(\lambda+2\mu))\right]\right.
\\[1pt]
\\
\left.+e^{\beta_2\xi_3}\beta_1\left[(\eta_1^2+\eta_1^2)\lambda(-2\beta_2^2\eta_1^2\lambda+\eta_1^2(-\eta_1^2+\eta_2^2)\lambda+\beta_2^4(\lambda+2\mu))\right.\right.
\\[1pt]
\\
\left.\left.+2\beta_1\beta_2\mu(2\eta_1^2(-\eta_1^2+\eta_2^2)\lambda-\beta_2^2(\eta_1^2+\eta_2^2)\lambda+\beta_2^4(\lambda+2\mu))\right.\right.
\\[1pt]
\\
\left.\left.+k_2^2\beta_2^2[(\eta_1^2+\eta_2^2)\lambda^2-\beta_1^2\lambda(\lambda+2\mu)+2\beta_1\beta_2\mu(\lambda+2\mu)]\right]\right\}
\end{array}
$
\end{itemize}

\end{proof}
\section{Conclusion}
\label{sec5}
The analytical solution  for the 3D frequency-dependent displacement Green's function  was obtained herein for a homogeneous viscoelastic and isotropic half-space. This can be viewed as a platform for the development of innovative boundary integral equation techniques. Specifically, they can be used for the numerical solution of various problems in mechanics under transient conditions which involve the half-space containing different type of heterogeneities that are either of natural or of artificial origin.

\textbf{Acknowledgment}

This work was partially supported by Grant No BG05M2OP001-1.001-0003, financed by the Bulgarian Science and Education for Smart Growth Operational Program (2014-2020) and co-financed by the European Union through the European Structural and Investment Funds.


\begin{thebibliography}{10}
\providecommand{\url}[1]{{#1}}
\providecommand{\urlprefix}{URL }
\expandafter\ifx\csname urlstyle\endcsname\relax
  \providecommand{\doi}[1]{DOI~\discretionary{}{}{}#1}\else
  \providecommand{\doi}{DOI~\discretionary{}{}{}\begingroup
  \urlstyle{rm}\Url}\fi

\bibitem{DMW19}
Dineva, P., Manolis, G., Wuttke, F.: Fundamental solutions in 3\textsc{D}
  elastodynamics for the \textsc{BEM}: A review.
\newblock Eng. Anal.Bound. Elem. \textbf{105}, 47--69 (2019)

\bibitem{Do93}
Dominguez, J.: Boundary Elements in Dynamics.
\newblock Computational Mechanics Publications, Southampton (1993)

\bibitem{ES75}
Eringen, A.C., Suhubi, E.S.: Elasto--dynamics. vol. 2: Linear Theory.
\newblock Academic Press, New York (1975)

\bibitem{Jo74}
Johmson, L.R.: Green`s function for \textsc{L}amb`s problem.
\newblock Geophys. J. Int. \textbf{37}(1), 99--131 (1974)

\bibitem{Ka06}
Kausel, E.: Fundamental Solutions in Elastodynamics: A Compendium.
\newblock Cambridge University Press, Cambridge (2006)

\bibitem{La04}
Lamb, H.: On the propagation of tremors over the surface of an elastic solid.
\newblock Philos. Trans. R. Soc. London A \textbf{203}, 1--42 (1904)

\bibitem{Ma10}
Mainardi, F.: Fractional Calculus and Waves in Linear Viscoelasticity: An
  Introduction to Mathematicam Models.
\newblock Imperial College Press, London, UK (2010)

\bibitem{MM16}
Makrou, A.A., Manolis, G.D.: A fractional derivative \textsc{Z}ener model for
  the numerical simulation of base isolated structures.
\newblock Bull. Earthquake Eng. \textbf{14}(1), 283--295 (2016)

\bibitem{MDRW17}
Manolis, G.D., Dineva, P.S., Rangelov, T.V., Wuttke, F.: Seismic Wave
  Propagation in Non--Homogeneous Elastic Media by Boundary Elements.
\newblock Solid Mechanics and its Applications, v. 240, Springer Int. Publ.,
  Cham, Switzerland (2017)

\bibitem{MS96}
Manolis, G.D., Shaw, R.: Green's function for a vector wave equation in mildly
  heterogeneous continuum.
\newblock Wave Motion \textbf{24}, 59--83 (1996)

\bibitem{Mi60}
Midlin, R.D.: Waves and vibrations in isotropic, elastic plates.
\newblock In: J.N. Goodier, N.J. Hoff (eds.) Structural Mechanics. Pergamon
  Press, New York (1960)

\bibitem{Pa87}
Pak, R.Y.S.: Asymmetric wave propagation in an elastic half-space by a method
  of potentials.
\newblock J. Appl. Mech. \textbf{54}(1), 121--126 (1995)

\bibitem{RDM18}
Rangelov, T.V., Dineva, P.S., Manolis, G.D.: Dynamic response of a cracked
  viscoelastic anisotropic plane using boundary elements and fractional
  derivatives.
\newblock J. Theor. Appl. Mech. \textbf{48}(2), 24--49 (2018)

\bibitem{TAG01}
Tadeu, A., Antonio, J., Godinho, L.: Green`s function for two and a half
  dimensional elastodynamic problems in a half-space.
\newblock Comput. Mech. \textbf{27}, 484--491 (2001)

\bibitem{Vl71}
Vladimirov, V.: Equations of Mathematical Physics.
\newblock Marcel Dekker, Inc., New York (1971)

\bibitem{YP16}
Yuan, H., Pan, Z.: Discussion on the time-harmonic elastodynamic half-space
  \textsc{G}reen`s function obtained by superposition.
\newblock Math. Probl. Eng. \textbf{ID 2717810}, 1--7 (2016)

\bibitem{ZG98}
Zhang, C., Gross, D.: On Wave Propagation in Elastic Solids with Cracks.
\newblock Comput. Mech. Publ., Southampton (1998)

\end{thebibliography}
\end{document}